\documentclass[11pt]{amsart}
\usepackage{amssymb, amsmath, amsfonts, mathrsfs, mathtools,epsfig,graphicx,color}
\usepackage{caption}
\usepackage{hyperref}
\usepackage[english]{babel}
\usepackage{enumitem}
\usepackage{bm}
\usepackage[normalem]{ulem}

\usepackage{natbib}
\setcitestyle{authoryear,open={(},close={)}}

\definecolor{darkblue}{rgb}{.1, 0.1,.8}
\definecolor{darkgreen}{rgb}{0,0.8,0.2}
\definecolor{darkred}{rgb}{.8, .1,.1}

\newtheorem{lemma}{Lemma}[section]

\newtheorem{theorem}[lemma]{Theorem}

\newtheorem{proposition}[lemma]{Proposition}
\newtheorem{definition}[lemma]{Definition}
\newtheorem{corollary}[lemma]{Corollary}
\newtheorem{example}[lemma]{Example}
\newtheorem{exercise}[lemma]{Exercise}
\newtheorem{remark}[lemma]{Remark}
\newtheorem{fig}[lemma]{Figure}
\newtheorem{tab}[lemma]{Table}

\newcommand{\bth}{\begin{theorem}}
\newcommand{\ethe}{\end{theorem}}

\newcommand{\bre}{\begin{remark}\em }
\newcommand{\ere}{\end{remark}}

\newcommand{\ble}{\begin{lemma}}
\newcommand{\ele}{\end{lemma}}

\newcommand{\bde}{\begin{definition}}
\newcommand{\ede}{\end{definition}}
\newcommand{\bco}{\begin{corollary}}
\newcommand{\eco}{\end{corollary}}

\newcommand{\bpr}{\begin{proposition}}
\newcommand{\epr}{\end{proposition}}

\newcommand{\bexer}{\begin{exercise}}
\newcommand{\eexer}{\end{exercise}}

\newcommand{\bexam}{\begin{example}}
\newcommand{\eexam}{\end{example}}

\newcommand{\bfi}{\begin{fig}}
\newcommand{\efi}{\end{fig}}

\newcommand{\btab}{\begin{tab}}
\newcommand{\etab}{\end{tab}}

\newcommand{\beao}{\begin{eqnarray*}}
\newcommand{\eeao}{\end{eqnarray*}\noindent}

\newcommand{\beam}{\begin{eqnarray}}
\newcommand{\eeam}{\end{eqnarray}\noindent}

\newcommand{\beqq}{\begin{equation}}
\newcommand{\eeqq}{\end{equation}\noindent}

\newcommand{\bce}{\begin{center}}
\newcommand{\ece}{\end{center}}

\newcommand{\barr}{\begin{array}}
\newcommand{\earr}{\end{array}}

\newcommand{\eqd}{\stackrel{d}{=}}

\newcommand{\vague}{\stackrel{\lower0.2ex\hbox{$\scriptscriptstyle
                    \it{v} $}}{\rightarrow}}
\newcommand{\weak}{\stackrel{\lower0.2ex\hbox{$\scriptscriptstyle
                    \it{w} $}}{\rightarrow}}
\newcommand{\what}{\stackrel{\lower0.2ex\hbox{$\scriptscriptstyle
                    \it{\hat{w}} $}}{\rightarrow}}

\newcommand{\bdis}{\begin{displaymath}}
\newcommand{\edis}{\end{displaymath}\noindent}

\newcommand{\N}{\mathbb{N}}
\renewcommand{\P}{\mathbb P}
\newcommand{\R}{\mathbb{R}}

\newcommand{\vep}{\varepsilon}

\def\1{\ensuremath{\mathrm{1}\hspace{-.35em} \mathrm{1}}} 

\def\E{{\mathbb E}}

\def\L{\mathbb{L}}
\def\N{\mathbb{N}}
\def\P{{\mathbb{P}}}
\def\R{\mathbb{R}}
\def\Z{\mathbb{Z}}

\renewcommand{\le}{\ensuremath{\leqslant}}
\renewcommand{\ge}{\ensuremath{\geqslant}}

\newcommand{\introo}[2]{{\left]{#1,\,#2\,}\right[\kern1pt}}

\newcommand{\intrfo}[2]{{\left[{#1,\,#2}\right[\kern1pt}}

\begin{document}

\title[Moment conditions for random coefficient AR($\infty$)]{Moment conditions for random coefficient AR($\infty$) under non-negativity assumptions}
\author{Pascal Maillard}
\address{Institut de Mathématiques de Toulouse\\
Université Toulouse 3 Paul Sabatier\\
118 Route de Narbonne\\
31062 Toulouse Cedex 9\\
France}
\thanks{PM is affiliated to Institut de Mathématiques de Toulouse (IMT), Université de Toulouse, CNRS UMR5219 and Institut Universitaire de France. Supported in part
by grants ANR-20-CE92-0010-01 (REMECO project) and ANR-11-LABX-0040 (ANR program “Investissements d’Avenir”).}
\email{firstname.lastname@math.univ-toulouse.fr}
\urladdr{https://www.math.univ-toulouse.fr/~pmaillar}
\author{Olivier Wintenberger}
\address{LPSM\\
Sorbonne Université\\
4 place Jussieu\\
75005 Paris\\
France}
\address{
Wolfgang Pauli Institute\\
Oskar-Morgenstern-Platz 1\\
A-1090 Wien \\
Austria}
\thanks{OW acknowledges support of the French Agence
Nationale de la Recherche (ANR) under reference ANR20-CE40-0025-01 (T-REX project). }
\email{firstname.lastname@sorbonne-universite.fr}
\urladdr{http://wintenberger.fr}
\date{July 14, 2022}

\begin{abstract}
    We consider random coefficient autoregressive models of infinite order (AR($\infty$)) under the assumption of non-negativity of the coefficients. We develop novel methods yielding sufficient or necessary conditions for finiteness of moments, based on combinatorial expressions of first and second moments. The methods based on first moments recover previous sufficient conditions by \cite{DW} in our setting. The second moment method provides in particular a necessary and sufficient condition which is different, but shown to be equivalent, to the classical criterion of \cite{Nicholls1982} in the case of finite order random coefficient autoregressive models. We further illustrate our results through two examples.
\end{abstract}

\maketitle


{Let $(A_j)_{j\ge1}$ be a sequence of non-negative random variables and $B$ a non-negative random variable. We allow for arbitrary dependencies among them.}
Consider the recurrence 
\begin{equation}
\label{eq:AR_infty}
X_t = \sum_{j=1}^{\infty}A_{t,j} X_{t-j}+B_t\,,\qquad t\in \Z\,,
\end{equation}
{where $((A_{t,j})_{ j\ge 1}, B_t)_{t\in\Z}$ is an iid sequence with generic element $((A_j)_{j\ge1},B)$.} In the case where there exists $p\in\N$ such that $A_j = 0$ for all $j>p$, this recurrence is known in the literature under the name \emph{random coefficient autoregressive model of $p$-th order}, symbolically $AR(p)$ in \cite[Section 4.4.9]{BDM}, or also \emph{random difference equation of $p$-th order} in \cite{Kesten1973}. It is natural to call the general recurrence the \emph{random coefficient autoregressive model of infinite order}, symbolically $AR(\infty)$. 
\par
Using a backward iteration starting from zero (\cite{DF}), one can construct a non-anticipative stationary solution of \eqref{eq:AR_infty}, whose {marginal distribution is the law of the following random variable:}
\begin{equation}
\label{eq:def_Xtilde2}
\tilde X \coloneqq \sum_{0=t_0<t_1< \cdots < t_n,\ n\ge 0} \tilde A_{t_0,t_1-t_0} \cdots \tilde A_{t_{n-1},t_n-t_{n-1}}B_{-t_n}\,,
\end{equation}
where $\tilde{A}_{t,j} \coloneqq A_{-t,j}$, $t\in\Z$, $j\ge 1$. By non-negativity, this quantity is always well-defined but may be infinite. The aim of this paper is to find explicit necessary and sufficient conditions for the existence of moments of $\tilde X$. We do not address here the question {of} whether the solution to \eqref{eq:AR_infty} thus constructed is unique and refer to \cite{DW} for results in this direction. As usual we will always assume in the sequel that the random variables $B_t$ satisfy $\P(B_t=0)<1$, $t\in \Z$, to avoid any degeneracy.


\subsection*{\texorpdfstring{The finite order case $AR(p)$, $p<\infty$}{The finite order case}} Let us recall what is known in the finite order case, following \cite{Kesten1973} and the recent book by \cite{BDM}.   When there exists $p\in\N$ such that $A_{1,j} = 0$ for all $j>p$, the recursion \eqref{eq:AR_infty} turns into the finite order $AR(p)$ equation
\[
X_t = A_{t,1} X_{t-1}+\cdots+A_{t,p} X_{t-p}+B_t\,,\qquad t\in \Z\,.
\]
One typically {rephrases} it as a random difference equation of first order on $p\times p$ matrices: Define
\[
\boldsymbol A_t := 
\begin{pmatrix}
A_{t,1} & A_{t,2} & \cdots & A_{t,p}\\
1 &  &  & \\
 & \ddots & &\\
 & & 1 &
\end{pmatrix},
\qquad
\boldsymbol B_t :=
\begin{pmatrix}
B_t\\
0 \\
\vdots\\
0
\end{pmatrix},
\qquad
\boldsymbol X_t :=
\begin{pmatrix}
X_t\\
\vdots\\
X_{t-p+1}
\end{pmatrix},
\]
with all non-appearing entries equal to zero in the definition of $\boldsymbol A_t$. {We also denote by $\boldsymbol A$ the matrix defined correspondingly with $(A_1,\ldots,A_p)$ instead of $(A_{t,1},\ldots, A_{t,p})$ and denote $\boldsymbol B$ accordingly.} Then \eqref{eq:AR_infty} is equivalent to
\begin{equation}
\label{eq:AR_matrix}
\boldsymbol X_t = \boldsymbol A_t \boldsymbol X_{t-1} + \boldsymbol B_t.
\end{equation}
This equation has been studied by many authors starting from \cite{Kesten1973}. Let us recall how it is typically solved. First, if one seeks stationary solutions of \eqref{eq:AR_infty} one needs to assume that the top Lyapunov exponent of the matrix $\boldsymbol A$ is negative, i.e.,
\begin{equation}\label{eq:toplyapunov}
\lim_{n\to\infty} \frac 1 n \log \|\boldsymbol A_n \cdots \boldsymbol A_1\| < 0,\quad \text{a.s.},
\end{equation}
where we can of course choose $\|\cdot\|$ to be any norm.
Next, one defines the following function \cite[(4.4.35)]{BDM}:
\begin{equation}\label{eq:toplyapunov2} 
h(\theta) = \lim_{n\to\infty} \frac1n\log \E\left[\|\boldsymbol A_n\cdots \boldsymbol A_1\|^\theta\right] 
\end{equation}
(existence of the limit follows from a submultiplicativity argument). The tail behavior of the stationary solution \eqref{eq:def_Xtilde2} depends on the form of this function. Most of the literature works under the assumptions where there exist $0<\alpha<\alpha'$, such that $h(\alpha) = 0$, $h(\alpha') < \infty$ and $\E[B^{\alpha'}]<\infty$. In this case, and under a further non-lattice assumption, there exists a stationary solution {with generic element $\boldsymbol X$ satisfying the distributional equation  
\begin{equation}
\label{eq:X=AX+B}
\boldsymbol X \stackrel{d}{=}\boldsymbol A\boldsymbol X +\boldsymbol B,
\end{equation}
with $(\boldsymbol A,\boldsymbol B)$ independent of $\boldsymbol X$. Furthermore, $\|\boldsymbol X\|$} is power law tailed of order $\alpha$, i.e. there exists $C>0$ such that\footnote{We use the notation $f(x)\sim g(x)$, $x\to \infty$ to mean that the two real-valued functions $f$ and $g$ satisfy the relation $\lim_{x\to \infty}f(x)/g(x)=1$.}
\[
\P(\|\boldsymbol X\| > x) \sim C\,x^{-\alpha},\qquad x\to\infty\,,
\] 
see \cite{BDM}. One easily deduces the moment properties $\E[\tilde X^\theta]<\infty$ for $\theta<\alpha$ and $\E[\tilde X^\theta]=\infty$ otherwise. In other words, using the convexity of the function $h$, the condition $h(\theta)<0$ is necessary and sufficient for the existence of moments of order $\theta>0$ for the $AR(p)$ model. The aim of the paper is to extend such necessary and sufficient conditions when $p=\infty$.

\par 

{The root equation $h(\alpha)=0$ is impossible to solve {explicitly} for general stochastic matrices $\boldsymbol A$. Even the simpler top-Lyapunov condition \eqref{eq:toplyapunov} rarely turns into an explicit condition on the law of the random vector $(A_{1},\ldots,A_{p})$.}
The specificity of the matrix $\boldsymbol A$ in the $AR(p)$ model allows \cite{Nicholls1982} 
to provide an explicit sufficient condition for the existence of a stationary solution with finite second moment, in the case of centered coefficients. Denoting by $\rho$ the spectral radius and $\otimes$  the tensor product, the sufficient condition of \eqref{eq:toplyapunov} is 
\begin{equation}\label{eq:nqcond}
\rho(\E[\boldsymbol A \otimes \boldsymbol A])  < 1\,.
\end{equation}
\cite[Corollary 2.2.2]{Nicholls1982} also show that this assumption is necessary and sufficient for having finite moments of order $2$ when $B$ has finite moments of order $2$ and is independent of $\boldsymbol A$. Thus the condition $\rho(\E[\boldsymbol A \otimes \boldsymbol A])  < 1$ is equivalent to the condition $h(2)<0$  in the case of centered coefficients $B_{t}$. \cite{tuan1986mixing} extended the equivalence to a more general setting including non-negative coefficients as considered here. In this article, we are interested in similarly explicit conditions when $p=\infty$, in the case of non-negative coefficients.


\subsection*{\texorpdfstring{Back to $AR(\infty)$}{Back to AR(infty)}}

Very little is known in the infinite-order case $p=\infty$, when the recursion \eqref{eq:AR_infty} is satisfied with $\P(A_{t,j}>0)>0$ for infinitely many $j\ge 1$. In the language of autoregressive processes, this model is denoted by $AR(\infty)$ and is an example of an autoregressive process of ``infinite memory''. \cite{DW} studied non-linear processes with infinite memory, which includes the $AR(\infty)$ as a special case and are defined by
\[
X_t=F(X_{t-1},X_{t-2},\ldots;\xi_t)\,,\qquad t\in \Z\,,
\]
where the iid process $(\xi_t)$ can be taken equal to $((A_{t,j})_j, B_t)$. They gave sufficient conditions for finite moments of order $\theta$ using a coupling approach in $\L^\theta$, $\theta\ge 1$. 
In particular, for the model \eqref{eq:AR_infty}, the contraction condition (3.2) in \cite{DW}, for $\theta \ge 1$, turns into the condition
\[
\tilde \phi_1(\theta) < 0,
\]
where 
\begin{equation}
\label{eq:phi_1_tilde}
\tilde \phi_1(\theta):=\log \sum_{j=1}^\infty \E\big[ A_{j}^{\theta}\big]^{1/\theta}\,,\qquad \theta> 0\,.
\end{equation}
Thus, if $\tilde\phi_1(\theta)<0$ and $\E[B^\theta]<\infty$ for some $\theta \ge 1$, Theorem~3.1  in \cite{DW} ensures the existence of a stationary solution $X$ admitting finite moments of order $\theta$.    {We are not aware of any other explicit condition appearing in the literature for the general $AR(\infty)$ model.} 

\subsection*{Relation with the smoothing transform}

{Equation \eqref{eq:X=AX+B}, or rather, its restriction to the first coordinate,} looks similar to an equation known in the literature as the \emph{fixed point equation of the (non-homogeneous) smoothing transform}. This is the following distributional equation {(the unknown being the law of the random variable $Y$)}
\begin{align}
\label{eq:smoothing_transform}
Y \eqd \sum_{j=1}^\infty A_{j} Y^{(j)} + B,
\end{align}
 {where $(Y^{(j)})_{j\ge1}$ are iid copies of the non-negative random variable $Y$, independent of $((A_j)_{j\ge1},B)$.} Define
\begin{align}
\label{eq:phi_1}
\phi_1(\theta):= \log \sum_{j=1}^\infty  \E\big[A_{j}^{\theta}\big]\,,\qquad \theta>0\,.
\end{align}
It is known that a solution to \eqref{eq:smoothing_transform} exists if $\phi_1(0) > 0$ and for some $\theta\in(0,1]$,  $\E[B^\theta]<\infty$ and $\phi_1(\theta) < 0$. Furthermore, this criterion is close to being optimal \cite[Section 5.2.4]{BDM}. In this case, if there exists $\alpha\ge \theta$ such that $\phi_1(\alpha) = 0$ and some non-arithmeticity condition is met, then $\P(Y > x)\sim C x^{-\alpha}$ as $x\to\infty$, for some constant $C\in (0,\infty)$. Note that the characteristic exponent $\alpha$ does not depend on the dependencies between the $A_{j}$'s, only on the marginal laws.

\subsection*{Description of our results} {Our results provide moment properties for the non-negative random variable $\tilde X$ defined in \eqref{eq:def_Xtilde2} under conditions involving moments of the (non-negative) random coefficients $A_{t,j}$ or their generic elements $A_j$, $j\ge 1$. Recall that $\tilde X$ is distributed as the (marginal distribution of the) stationary solution of the recurrence \eqref{eq:AR_infty}.}
Our first result is an extension of the above-mentioned sufficiency condition from \cite{DW} to the case $\theta \le 1$. More precisely we prove that $\phi_1(\theta)<0$ is a sufficient condition for $\E[\tilde X^\theta]<\infty$ when $0<\theta\le 1$. Then we show that $\phi_1(\theta)<0$ is a necessary condition for $\E[\tilde X^\theta]<\infty$ when $\theta\ge 1$. Finally we also show that $\tilde\phi_1(\theta)<0$ is a necessary condition for $\E[\tilde X^\theta]<\infty$ when $0<\theta<1$. Since the functions $\phi_1$ and $\tilde \phi_1$ coincide at $\theta=1$, we deduce that $\phi_1(1)=\tilde \phi_1(1)<0$ is a necessary and sufficient condition for the existence of moments of order $1$ when $\E[B]<\infty$. We thus obtain necessary conditions as well as sufficient conditions for any $\theta>0$, and these are asymptotically sharp when $\theta$ is close to one. The special role of $\theta=1$ is not surprising in view of the linearity of the model and of the expectation. 
\par 
Our main contributions are necessary conditions and sufficient conditions of moments that are better than the previous ones for $\theta \ge 2$ and for some values of $\theta$ in the interval $[1,2]$. Furthermore, they are asymptotically sharp when $\theta$ is close to $2$. 
To do so we introduce in Section~\ref{sec:secondmom} two other functions $\phi_2$ and $\tilde \phi_2$ that have the following properties (we omit some extra conditions below to simplify the presentation): 
\begin{itemize}
    \item for $0<\theta<2$ the conditions $ \phi_2(\theta)<0$ and $\tilde\phi_2(\theta)<0$ are sufficient and necessary, respectively, for the existence of moments $\E[\tilde X^\theta]<\infty$,
    \item for $\theta>2$ the conditions $ \phi_2(\theta)<0$ and $\tilde\phi_2(\theta)<0$ are necessary and sufficient, respectively, for the existence of moments $\E[\tilde X^\theta]<\infty$,
    \item for $\theta=2$ the condition $ \phi_2(\theta)=\tilde\phi_2(\theta)<0$ is necessary and sufficient  for the existence of moments $\E[\tilde X^2]<\infty$.
\end{itemize}
The statement above is only a partial summary omitting extra conditions in some cases. In particular, conditions on the functions $\phi_1$ and $\tilde\phi_1$ are required for ensuring the existence of moments; see Theorem \ref{th:second_moment} for a precise statement. {The functions $\phi_1$, $\tilde \phi_1$, $\phi_2$ and $\tilde \phi_2$ do not depend on the law of the non-negative random variables $B_t$. In fact, our sufficient or necessary conditions for finiteness of moments of $\tilde X$ apply both in situations when $\tilde X$ has a heavier tail than $B_t$ (under moment conditions on $B_t$) and in situations when the moment properties of $\tilde X$ and $B_t$ are similar.} In the latter case, one could also apply Theorem 3.1 of \cite{hult2008tail} to obtain the precise tail properties of $\tilde X$ under moments conditions on $((A_{t,j})_{ j\ge 1})_{t\in \Z}$ and regular variation of the tail of $B_t$.

We furthermore verify that {the necessary and sufficient condition $ \phi_2(\theta)=\tilde\phi_2(\theta)<0$} coincides with the classical one \eqref{eq:nqcond} of \cite{Nicholls1982} in the case $\theta=2$ when $p<\infty$, although it is of a very different form. Indeed our approach uses linearity together with combinatorics on pairs instead of a matrix representation as the one of \cite{Nicholls1982}. The equivalence stated in Theorem \ref{th:comparison} extends the classical necessary and sufficient condition \eqref{eq:nqcond} to the stationary solution of the AR($p$) model \eqref{eq:AR_matrix}.

We mention that our approach also covers models of the form
\begin{equation}
\label{eq:AR_infty2}
X_t = \sum_{j=1}^{\infty}A_{t-j,j} X_{t-j}+B_{t-1}\,,\qquad t\in \Z\,,
\end{equation}
{with $((A_{t,j})_{j\ge 1},B_t)_{t\in\Z}$ defined as above, assuming that the $B_t$'s are independent of the $A_{t,j}$'s.}
The non-anticipative solution of \eqref{eq:AR_infty2}  is a {\em predictable} process with respect to the natural filtration 
\[
\mathscr F_t := \sigma(((A_{s,j})_{j\ge 1}, B_s); s\le t).
\]
The volatility process $(\sigma_t^2)$ of a GARCH(1,1) model $\vep_t=Z_t \sigma_t$ for iid $(Z_t)$ is a typical example. It satisfies the recursion 
\[
\sigma_t^2=\omega + \alpha_1 \vep_{t-1}^2+\beta_1 \sigma_{t-1}^2\,,\qquad t\in \Z\,.
\]
After expanding, this turns into a  recursion of infinite order of the form \eqref{eq:AR_infty2}
\[
\sigma_t^2=\dfrac{\omega}{1-\beta_1}+\alpha_1\sum_{j=1}^\infty \beta_1^{j-1}Z_{t-j}^2  \sigma_{t-j}^2\,,\qquad t\in\Z\,.
\]
We check that for the volatility process of a GARCH(1,1) model our conditions provide necessary and sufficient conditions of finite order moments. In particular we recover the optimal second-order moment condition of the GARCH(1,1) volatility using our approach. Notice that infinite memory models (AR and ARCH) are necessary for writing the invertible form (depending only on the past observations $\varepsilon_{t-1},\varepsilon_{t-2},\ldots$) of some finite memory models such as ARMA or GARCH.

There is a natural extension of some our methods to higher moments of order $k\ge 3$. However, while this easily yields necessary criteria for the finiteness of moments, we were not able to obtain simple sufficient criteria in that case. See Section~\ref{sec:higher-moments} for details.

Finally, notice also that sufficient conditions for finiteness of moments for some non-linear models (such as Galton-Watson process with immigration, see \cite{basrak2013heavy}) can be deduced from Theorem \ref{th:second_moment} using a stochastic domination argument similar as in \cite{DW}. 


%
%

\section{First moment -- comparison to smoothing transform}
\label{sec:first_moment}

{As in the introduction, let $(A_j)_{j\ge1}$ be a sequence of non-negative random variables and $B$ a non-negative random variable. We allow for arbitrary dependencies among them (we only require that $B$ is independent of the family $(A_j)_{j\ge1}$ under assumption \ref{A2} below).}
 In the remainder of the paper, we let $(X_t)_{t\ge0}$ be the solution to the equation \eqref{eq:AR_infty} starting from zero, i.e.
\begin{equation}
\label{eq:AR_infty_simple}
\begin{cases}
X_t = \sum_{j=1}^{\infty}A_{t,j} X_{t-j}+  B_t,\quad t\ge0 \\
X_{-1} = X_{-2} = \cdots = 0,
\end{cases}
\end{equation}
where the family $((A_{t,j})_{j\ge1},B_t)_{t\in\Z}$ satisfies one of the following two assumptions:
\begin{enumerate}[label=\textbf{(A\arabic*)}]
\item \label{A1} $((A_{t,j})_{j\ge1},B_t)_{t\in\Z}$ are iid copies of the random sequence $((A_j)_{j\ge1},B)$,\quad or
\item \label{A2} {$((A_{t+j,j})_{j\ge1})_{t\in\Z}$ are iid copies of the random sequence $(A_j)_{j\ge1}$. Furthermore, $B$ is independent of the family $(A_j)_{j\ge1}$.}
\end{enumerate}
{Under \ref{A2}, we set $A'_{t,j} \coloneqq A_{t+j,j}$ for any $t\in\Z$ and $j\ge1$, so that $A_{t,j} = A'_{t-j,j}$ for any $t\in\Z$ and $j\ge1$.
Note that under \ref{A1} or \ref{A2},} the sequence of random vectors $((A_{t,j})_{j\ge1})_{t\in\Z}$ is stationary in $t$, i.e.~translation-invariant in law. Under \ref{A1} it is also reversible, i.e.~invariant in law under the change of index $t\mapsto -t$, however, under \ref{A2} it is not in general. In order to express the distribution of the solution of \eqref{eq:AR_infty_simple} using backward iterates, we therefore define
\begin{equation}
\label{eq:Atilde_def}
\tilde{A}_{t,j} \coloneqq A_{-t,j},\qquad t\in\Z,\ j\ge 1.
\end{equation}
\begin{remark}
\label{rem:independence}
Given $t,t'\in\Z$ and $j,j'\ge 1$, one readily checks that $\tilde{A}_{t,j}$ and $\tilde{A}_{t',j'}$ are independent if
\begin{itemize}
\item $t\ne t'$ (under  \ref{A1}), or
\item $t+j\ne t'+j'$ (under  \ref{A2}),
\end{itemize}
with obvious generalizations for larger collections of r.v. Moreover, $\tilde{A}_{t,j}$ and $B_{-t'}$ are independent if $t\ne t'$ under \ref{A1} and for all $t,t'$ under \ref{A2}.
\end{remark}

The following result is the main result of this section.
\begin{theorem}\ \label{th:first_moment}
Let $X_t$ be defined by \eqref{eq:AR_infty_simple} and let $\phi_1$ and $\tilde \phi_1$ be defined in \eqref{eq:phi_1} and \eqref{eq:phi_1_tilde}.
\begin{enumerate}
\item As $t\to\infty$, $X_t$ converges in law to a (possibly infinite) limit $X$.
\item \label{firstassump} Assume $\phi_1(\theta) < 0$ and $\E[B^\theta]<\infty$ for some $\theta\in(0,1]$. Then $\E[X^\theta] < \infty$.
\item Assume $\phi_1(\theta) \ge 0$ or $\E[B^\theta]=\infty$ for some $\theta \ge 1$. Then $\E[X^{\theta}] = \infty$.
\item Assume $\tilde \phi_1(\theta) <0$ and $\E[B^\theta]<\infty$ for some $\theta \ge 1$. Then $\E[X^{\theta}] < \infty$.
\item Assume $\tilde \phi_1(\theta) \ge 0$ or $\E[B^\theta]=\infty$ for some $0<\theta < 1$. Then $\E[X^{\theta}] = \infty$.
\end{enumerate}
\end{theorem}

\begin{remark}\label{rem:min}
Consider the case $\theta=1$. By definition, $\phi_1(1) = \tilde\phi_1(1)$. 
By Theorem~\ref{th:first_moment}, if $\E[B]<\infty$, then $\phi_1(1)=\tilde \phi_1(1)<0$ is a necessary and sufficient condition for $\E[X]<\infty$. This condition rewrites explicitly as 
\begin{equation*}
\sum_{j=1}^\infty \E[A_j]<1 
\end{equation*}
and does not depend on the dependence among the $A_j$.
\end{remark}
\begin{proof}
We begin by proving 1.
Unravelling the recurrence \eqref{eq:AR_infty_simple}, we get 
\begin{align*}
X_t 
&= \sum_{0\le t_0<t_1< \cdots < t_n = t,\ n\ge 0} A_{t_n,t_n-t_{n-1}} \cdots A_{t_1,t_1-t_0}B_{t_0}\\
&= \sum_{0=t_0<t_1< \cdots < t_n \le t,\ n\ge 0} A_{t-t_0,t_1-t_0} \cdots A_{t-t_{n-1},t_n-t_{n-1}}B_{t-t_n}\,.
\end{align*}
Using stationarity in $t$ and \eqref{eq:Atilde_def}, we get
\begin{align*}
X_t &\eqd \sum_{0=t_0<t_1< \cdots < t_n \le t,\ n\ge 0} \tilde A_{t_0,t_1-t_0} \cdots \tilde A_{t_{n-1},t_n-t_{n-1}}B_{-t_n}\\
&\eqqcolon \tilde X_t.
\end{align*}
By non-negativity of the $A$'s, the sequence $\tilde X_t$ is {non-decreasing} in $t$ and therefore converges almost surely as $t\to\infty$ to the (possibly infinite) limit {$\tilde X$ defined in \eqref{eq:def_Xtilde2}, namely,}
\begin{equation*}
\tilde X = \sum_{0=t_0<t_1< \cdots < t_n,\ n\ge 0} \tilde A_{t_0,t_1-t_0} \cdots \tilde A_{t_{n-1},t_n-t_{n-1}}B_{-t_n}.
\end{equation*}
Since $X_t\eqd \tilde X_t$ for every $t$, this shows that $X_t$ converges in law to a (possibly infinite) limit $X$. Furthermore, $X\eqd \tilde X$.

We now prove 2. showing that with $\theta$ from the assumption, i.e.~$\theta\in(0,1]$ and $\phi_1(\theta) < 0$, $\E[X^\theta] = \E[{\tilde X}^\theta] < \infty$ (note that this implies in particular that $X$ is finite almost surely).   By subadditivity of the function $x\mapsto x^\theta$, we have
\begin{align*}
\E[{\tilde X}^\theta] 
&\le \sum_{0=t_0<t_1< \cdots < t_n,\ n\ge 0} \E[(\tilde A_{t_0,t_1-t_0} \cdots \tilde A_{t_{n-1},t_n-t_{n-1}}B_{-t_n})^\theta].
\end{align*}
By independence (see Remark~\ref{rem:independence}),
\begin{align*}
\E[{\tilde X}^\theta] 
&\le \sum_{0=t_0<t_1< \cdots < t_n,\ n\ge 0} \E[A_{t_1-t_0}^\theta] \cdots \E[A_{t_n-t_{n-1}}^\theta]\E[B^\theta]\\
&= \sum_{i_1,\ldots,i_n\ge 1,\,n\ge 0} \E[A_{i_1}^\theta] \cdots \E[A_{i_n}^\theta]\E[B^\theta]\\
&= \E[B^\theta]\sum_{n=0}^\infty \left(\sum_{i=1}^\infty \E[A_i^\theta]\right)^n\\
&= \E[B^\theta]\sum_{n=0}^\infty e^{n\phi_1(\theta)}\\
&< \infty,
\end{align*}
since $\phi_1(\theta) < 0$ by hypothesis. This finishes the proof of 2.

The proof of 3. is similar. Let $\theta \ge 1$ and $\phi_1(\theta) \ge 0$ or $\E[B^\theta]=\infty$. By superadditivity of the function $x\mapsto x^{\theta}$, we have, similarly as above,
\begin{align*}
\E[{\tilde X}^{\theta}] &\ge  \E[B^\theta]\sum_{n=0}^\infty e^{n\phi_1(\theta)} = \infty.
\end{align*}

The assertion 4. follows from \cite{DW} in the cases where \ref{A1} is satisfied. We provide here an alternative proof that also works under \ref{A2}. Using the Minkowski inequality ($\theta\ge1$) and by independence (see Remark~\ref{rem:independence}),
\begin{align*}
\E[{\tilde X}^\theta]^{1/\theta} 
&\le \sum_{0=t_0<t_1< \cdots < t_n,\ n\ge 0} \E[(A_{t_1-t_0} \cdots  A_{t_n-t_{n-1}}B_{-t_n})^\theta]^{1/\theta}\\
&= \sum_{0=t_0<t_1< \cdots < t_n,\ n\ge 0} \E[A_{t_1-t_0}^\theta]^{1/\theta} \cdots  \E[A_{t_n-t_{n-1}}^\theta]^{1/\theta}\E[B^\theta]^{1/\theta}\\
&= \sum_{i_1,\ldots,i_n\ge 1,\,n\ge 0} \E[A_{i_1}^\theta]^{1/\theta} \cdots \E[A_{i_n}^\theta]^{1/\theta}\E[B^\theta]^{1/\theta}\\
&= \E[B^\theta]^{1/\theta}\sum_{n=0}^\infty \left(\sum_{i=1}^\infty \E[A_i^\theta]^{1/\theta}\right)^n\\
&= \E[B^\theta]^{1/\theta}\sum_{n=0}^\infty e^{n\tilde \phi_1(\theta)}\\
&< \infty,
\end{align*}
since $\tilde \phi_1(\theta) < 0$ by hypothesis.

The proof of 5. is similar to the preceding one, replacing the Minkowski inequality with the reverse Minkowski inequality holding for $0<\theta\le 1$ and non-negative random variables:
\begin{align*}
\E[{\tilde X}^\theta]^{1/\theta} 
&\ge \sum_{0=t_0<t_1< \cdots < t_n,\ n\ge 0} \E[(A_{t_1-t_0} \cdots  A_{t_n-t_{n-1}}B_{-t_n})^\theta]^{1/\theta}\,.
\end{align*}
The proof of point~5.\ then follows along the same lines as the proof of point~4.
\end{proof}

\section{Second moment -- a combinatorial formula}\label{sec:secondmom}

From this section on, we let $X$ be the random variable from the statement of Theorem~\ref{th:first_moment} and $\tilde X$ the random variable from \eqref{eq:def_Xtilde2}.
We now calculate the second moment $\E[X^2]$. For this, it is useful to introduce some notation, {which allows us to express the second moment and related quantities in a compact way}. Define the set of finite increasing integer-valued sequences starting at zero:
\[
\mathscr T = \left\{\boldsymbol t = (t_0,\ldots,t_n): n\ge 0,\ 0=t_0 < t_1 <\cdots < t_n\right\}.
\]
The trivial sequence is denoted by $\boldsymbol 0 = (0)$. For $\boldsymbol t = (t_0,\ldots,t_n) \in \mathscr T$, we write $n(\boldsymbol t) = n$ and $\tilde A_{\boldsymbol t} = \tilde A_{t_0,t_1-t_0} \cdots \tilde A_{t_{n-1},t_n-t_{n-1}}$, with $\tilde A_{\boldsymbol 0} = 1$ by convention. We also denote $\tilde B_{\boldsymbol t}:=B_{-t_n}$. With this notation, we have
\begin{equation}
\label{eq:Xtilde_new_notation}
\tilde X =   \sum_{\boldsymbol t  \in \mathscr T} \tilde A_{\boldsymbol t} \tilde B_{\boldsymbol t},
\end{equation}
and hence
\begin{align}
\label{eq:Xtilde_squared}
{\tilde X}^2 = \sum_{\boldsymbol s,\boldsymbol t  \in \mathscr T} \tilde A_{\boldsymbol s}\tilde B_{\boldsymbol s}\tilde A_{\boldsymbol t}\tilde B_{\boldsymbol t}.
\end{align}

Define the concatenation of a finite number of sequences $\boldsymbol t^1,\ldots,\boldsymbol t^k$ by
\[
\boldsymbol t^1\cdots \boldsymbol t^k \coloneqq (t^1_0,\ldots,t^1_{n(\boldsymbol t^1)},t^1_{n(\boldsymbol t^1)}+ t^2_1,\ldots,t^1_{n(\boldsymbol t^1)}+t^2_{n(\boldsymbol t^2)},\ldots,t^1_{n(\boldsymbol t^1)}+\cdots +t^k_{n(\boldsymbol t^k)}).
\]
Now define the following two sets of pairs of elements in $\mathscr T$:
\begin{align*}
\mathscr C &= \{(\boldsymbol s,\boldsymbol t)\in \mathscr T\times \mathscr T: n(\boldsymbol s)>0,\,n(\boldsymbol t) > 0,\,s_{n(\boldsymbol s)} = t_{n(\boldsymbol t)},\\
&\quad \quad s_i \ne t_j,\ 0<i<n(\boldsymbol s),\,0 < j < n(\boldsymbol t)\}\\
\mathscr O &= \{(\boldsymbol s,\boldsymbol t)\in \mathscr T\times \mathscr T: \forall 0<i\le n(\boldsymbol s),\,0 < j \le n(\boldsymbol t): s_i \ne t_j \}.
\end{align*}
We call the pairs in $\mathscr C$ and $\mathscr O$ \emph{closed} and \emph{open}, respectively. Note that the trivial pair is open by definition: $(\boldsymbol 0,\boldsymbol 0) \in \mathscr O$. It is easy to see that any pair $(\boldsymbol s,\boldsymbol t)\in \mathscr T\times \mathscr T$ can be written as a concatenation of a finite number of closed pairs and an open pair: for every $(\boldsymbol s,\boldsymbol t)\in \mathscr T\times \mathscr T$ there exists $k\in\N_0$, $(\boldsymbol s^1,\boldsymbol t^1),\ldots,(\boldsymbol s^k,\boldsymbol t^k)\in\mathscr C$ and $(\boldsymbol s^o, \boldsymbol t^o)\in \mathscr O$, such that
\[
(\boldsymbol s,\boldsymbol t) = (\boldsymbol s^1\cdots \boldsymbol s^k\boldsymbol s^o,\boldsymbol t^1\cdots \boldsymbol t^k\boldsymbol t^o).
\]
This corresponds to splitting the pair $(\boldsymbol s,\boldsymbol t)$ at the points where the two sequences intersect, see Figure~\ref{fig:decomposition} for an illustration. Since closed pairs only intersect at the start and end points and open pairs only at the starting point, this decomposition is unique.

\begin{figure}[ht]
\centering
\includegraphics[scale=1]{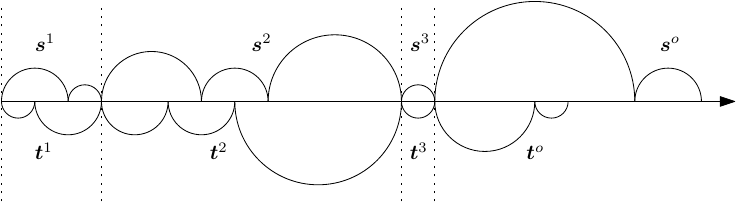}
\caption{Schematic illustration of the decomposition of a pair $(\boldsymbol s,\boldsymbol t)\in \mathscr T\times \mathscr T$ into finitely many (here, three) closed pairs $(\boldsymbol s^1,\boldsymbol t^1)$, $(\boldsymbol s^2,\boldsymbol t^2)$, $(\boldsymbol s^3,\boldsymbol t^3)$ and an open pair $(\boldsymbol s^o,\boldsymbol t^o)$.}
\label{fig:decomposition}
\end{figure}

In the same spirit as the functions $\phi_1$ and $\tilde \phi_1$ from the last section, we now introduce two functions $\phi_2$ and $\tilde \phi_2$, defined as follows:
\begin{align}
\label{eq:phi_2}
\phi_2(\theta)&:=\log \sum_{(\boldsymbol s,\boldsymbol t)\in\mathscr C}\E[\tilde A_{\boldsymbol s}^{\theta/2}\tilde A_{\boldsymbol t}^{\theta/2}]\,,\qquad \theta>0\\
\label{eq:phi_2_tilde}
\tilde \phi_2(\theta)&:=\log \sum_{(\boldsymbol s,\boldsymbol t)\in\mathscr C}\E[\tilde A_{\boldsymbol s}^{\theta/2}\tilde A_{\boldsymbol t}^{\theta/2}]^{2/\theta}\,,\qquad \theta>0\,.
\end{align}
Note that we have $\phi_2(\theta) \ge \phi_1(\theta)$  for all $\theta >0$, because the set $\mathscr C$ contains the closed pairs $((0,i),(0,i))$ for each $i\ge1$ and summing over these pairs only yields $\phi_1(\theta)$. 

The following theorem is the main result of this article:

\begin{theorem} \label{th:second_moment}
Let $\theta >0$. Let $\phi_1$, $\tilde \phi_1$, $\phi_2$, $\tilde \phi_2$ be defined as in \eqref{eq:phi_1}, \eqref{eq:phi_1_tilde}, \eqref{eq:phi_2}, \eqref{eq:phi_2_tilde}, respectively, let $X$ be defined as in Theorem~\ref{th:first_moment}.
\begin{enumerate}
\item Assume that either $\theta \ge 2$ and $\phi_2(\theta) \ge 0$, or $\theta \le 2$ and $\tilde \phi_2(\theta) \ge 0$. Then $\E[X^\theta] = \infty$.
\item Assume the following:
\begin{enumerate}
\item either: $\theta \le 2$, $\phi_1(\theta/2)<0$ and $\phi_2(\theta)<0$,\\
or: $\theta \ge 2$, $\tilde \phi_1(\theta/2)<0$ and $\tilde \phi_2(\theta)<0$,
\item $\E\left[\sum_{i,j\ge 1}\E[A_i^{\theta/2}A_j^{\theta/2}]^{\min(1,2/\theta)}\right] < \infty$ (only under \ref{A1}),
\item $\E[B^\theta]<\infty$.
\end{enumerate}
Then $\E[X^\theta] < \infty$.
\end{enumerate}
\end{theorem}

\begin{proof}
We start with proving the first statement (showing that $\E[X^\theta] = \infty$).
We first consider the case $\theta = 2$ and assume that $\phi_2(2) = \tilde \phi_2(2) > 0$. We wish to show that $\E[X^2] = \E[\tilde X^2] = \infty$. By the decomposition of arbitrary pairs $(\boldsymbol s,\boldsymbol t)$ into a concatenation of closed pairs and an open pair, {we obtain from \eqref{eq:Xtilde_squared} the following expression:}
\begin{align}
\label{eq:tilde_X_square}
{\tilde X}^2 &= \sum_{k=0}^\infty \sum_{(\boldsymbol s^1,\boldsymbol t^1),\ldots,(\boldsymbol s^k,\boldsymbol t^k)\in\mathscr C} \sum_{(\boldsymbol s^o, \boldsymbol t^o)\in \mathscr O} \tilde A_{\boldsymbol s^1\cdots \boldsymbol s^k\boldsymbol s^o}\tilde B_{ \boldsymbol s^1\cdots \boldsymbol s^k\boldsymbol s^o}\tilde A_{\boldsymbol t^1\cdots \boldsymbol t^k\boldsymbol t^o} \tilde B_{ \boldsymbol t^1\cdots \boldsymbol t^k\boldsymbol t^o}.
\end{align}

We now claim the following:
\begin{multline}\label{eq:identity}
\E[\tilde A_{\boldsymbol s^1\cdots \boldsymbol s^k\boldsymbol s^o}\tilde B_{ \boldsymbol s^1\cdots \boldsymbol s^k\boldsymbol s^o}\tilde A_{\boldsymbol t^1\cdots \boldsymbol t^k\boldsymbol t^o} \tilde B_{ \boldsymbol t^1\cdots \boldsymbol t^k\boldsymbol t^o}]\\
= \E[\tilde A_{\boldsymbol s^1}\tilde A_{\boldsymbol t^1}]\cdots \E[\tilde A_{\boldsymbol s^k}\tilde A_{\boldsymbol t^k}]\E[\tilde A_{\boldsymbol s^o} \tilde B_{ \boldsymbol s^o}\tilde A_{\boldsymbol t^o} \tilde B_{  \boldsymbol t^o}].
\end{multline}
To see this, set $\tau_0 = 0$ and for $i=1,\ldots,k$:
\[
\tau_i \coloneqq s^1_{n(s^1)} + \cdots + s^i_{n(s^i)} = t^1_{n(t^1)} + \cdots + t^i_{n(t^i)},
\]
where the equality comes from the definition of closed pairs. In Figure~\ref{fig:decomposition}, the $\tau_i$ correspond to the abscissae of the dotted lines. We now write
\begin{multline*}
\tilde A_{\boldsymbol s^1\cdots \boldsymbol s^k\boldsymbol s^o}\tilde A_{\boldsymbol t^1\cdots \boldsymbol t^k\boldsymbol t^o} = \Pi_1\cdots\Pi_k\Pi_o,\quad \text{where}\\
\Pi_i = \frac{\tilde A_{\boldsymbol s^1\cdots \boldsymbol s^i}\tilde A_{\boldsymbol t^1\cdots \boldsymbol t^i}}{\tilde A_{\boldsymbol s^1\cdots \boldsymbol s^{i-1}}\tilde A_{\boldsymbol t^1\cdots \boldsymbol t^{i-1}}},\quad i=1,\ldots,k,\\ \text{and}\quad \Pi_o = \frac{\tilde A_{\boldsymbol s^1\cdots \boldsymbol s^k \boldsymbol s^o}\tilde B_{ \boldsymbol s^1\cdots \boldsymbol s^k\boldsymbol s^o}\tilde A_{\boldsymbol t^1\cdots \boldsymbol t^k \boldsymbol t^o}\tilde B_{ \boldsymbol t^1\cdots \boldsymbol s^k\boldsymbol t^o}}{\tilde A_{\boldsymbol s^1\cdots \boldsymbol s^k}\tilde A_{\boldsymbol t^1\cdots \boldsymbol t^k}}\,.
\end{multline*}
Now note that for every $i=1,\ldots,k$, $\Pi_i$ is a product of terms $\tilde A_{t,j}$ with
\begin{itemize}
\item $t\in \{\tau_{i-1},\ldots,\tau_{i}-1\}$, and 
\item $t+j \in \{\tau_{i-1}+1,\ldots,\tau_i\}.$
\end{itemize}
Furthermore, $\Pi_o$ is a product of terms $\tilde A_{t,j}$ and $B_{-t}$ with
\begin{itemize}
\item $t\in \{\tau_{k},\tau_k+1,\ldots\}$, and 
\item $t+j \in \{\tau_k+1,\tau_k+2,\ldots\}.$
\end{itemize}
Hence, under either  \ref{A1} or \ref{A2}, the products $\Pi_1,\ldots,\Pi_k,\Pi_o$ are independent, see Remark~\ref{rem:independence}. Furthermore, by stationarity of the sequence $((\tilde A_{t,j})_{j\ge 1})_{t\ge 0}$, we have $\E[\Pi_i] = \E[\tilde A_{\boldsymbol s^i}\tilde A_{\boldsymbol t^i}]$ for all $i=1,\ldots,k$ and $\E[\Pi_o] = \E[\tilde A_{\boldsymbol s^o}\tilde B_{\boldsymbol s^o}\tilde A_{\boldsymbol t^o}\tilde B_{\boldsymbol t^o}]$. This yields \eqref{eq:identity}.

%
%

Equations \eqref{eq:tilde_X_square} and \eqref{eq:identity} together now give 
\begin{align}
\label{eq:second_moment}
\E[X^2] = \left(\sum_{(\boldsymbol s,\boldsymbol t)\in\mathscr O}\E[\tilde A_{\boldsymbol s}\tilde B_{\boldsymbol s}\tilde A_{\boldsymbol t}\tilde B_{\boldsymbol t}]\right) \sum_{k=0}^\infty \left(\sum_{(\boldsymbol s,\boldsymbol t)\in\mathscr C}\E[\tilde A_{\boldsymbol s}\tilde A_{\boldsymbol t}]\right)^k .
\end{align}

Noting that
\[
\sum_{(\boldsymbol s,\boldsymbol t)\in\mathscr O}\E[\tilde A_{\boldsymbol s}\tilde B_{\boldsymbol s}\tilde A_{\boldsymbol t}\tilde B_{\boldsymbol t}] \ge \E[\tilde A_{\boldsymbol 0}\tilde B_{\boldsymbol 0}\tilde A_{\boldsymbol 0}\tilde B_{\boldsymbol 0}] = \E[B^2] > 0,
\]
and
\[
\sum_{(\boldsymbol s,\boldsymbol t)\in\mathscr C}\E[\tilde A_{\boldsymbol s}\tilde A_{\boldsymbol t}] = e^{\phi_2(2)} = e^{\tilde \phi_2(2)} \ge 1,
\]
we get from \eqref{eq:second_moment} that $\E[X^2] = \infty$. This proves the first statement of the theorem in the case $\theta = 2$.

Now assume $\theta \ge 2$ and $\phi_2(\theta) \ge 0$. We use superadditivity of the function $x\mapsto x^{\theta/2}$ and {\eqref{eq:Xtilde_new_notation}} to bound
\begin{align*}
\E[X^\theta] \ge \E\left[\left(\sum_{\boldsymbol t\in\mathscr T} \tilde A_{\boldsymbol t}^{\theta/2}\tilde B_{\boldsymbol t}^{\theta/2}\right)^2\right],
\end{align*}
and then use the case $\theta = 2$ with $A_{t,j}$ and $B_j$ replaced by $A_{t,j}^{\theta/2}$ and $B_j^{\theta/2}$ for all $t,j$.

Now assume that $\theta \le 2$ and $\tilde \phi_2(\theta) \ge 0$. Similarly to the proof of \eqref{eq:second_moment}, but using moreover the reverse Minkowski's inequality, we get

\begin{align}
\label{eq:second_moment_theta}
&\quad\E[X^\theta]^{2/\theta} = \E[(X^2)^{\theta/2}]^{2/\theta}\\
\nonumber
&\ge \left(\sum_{(\boldsymbol s,\boldsymbol t)\in\mathscr O}\E[\tilde A_{\boldsymbol s}^{\theta/2}\tilde B_{\boldsymbol s}^{\theta/2}\tilde A_{\boldsymbol t}^{\theta/2}\tilde B_{\boldsymbol t}^{\theta/2}]^{2/\theta}\right) \sum_{k=0}^\infty \left(\sum_{(\boldsymbol s,\boldsymbol t)\in\mathscr C}\E[\tilde A_{\boldsymbol s}^{\theta/2}\tilde A_{\boldsymbol t}^{\theta/2}]^{2/\theta}\right)^k\\
\nonumber
&\ge \E[B^\theta]^{2/\theta} \sum_{k=0}^\infty \left(e^{\tilde \phi_2(\theta)}\right)^k\\
\nonumber
&= \infty.
\end{align}
This yields the first statement {of the theorem} in the case $\theta \le 2$ and thus finishes the proof of the first statement. 





We now turn to the proof of the second statement (finiteness of $\E[X^\theta]$). Again, we first consider the case $\theta = 2$ (note that $\phi_1(1) = \tilde \phi_1(1)$ and $\phi_2(2) = \tilde \phi_2(2)$ by definition). By assumption (a), the geometric series in \eqref{eq:second_moment} is finite. It suffices to show that $\sum_{(\boldsymbol s,\boldsymbol t)\in\mathscr O}\E[\tilde A_{\boldsymbol s}\tilde B_{\boldsymbol s}\tilde A_{\boldsymbol t}B_{\boldsymbol t}]$ is finite as well.

We separately consider the cases \ref{A1} and \ref{A2}. We start with the simpler case \ref{A2}. Let $(\boldsymbol s,\boldsymbol t)\in\mathscr O$. By definition, $\tilde A_{\boldsymbol s}$ is a product of terms $\tilde A_{t,j}$ with $t+j \in S\coloneqq \{s_1,\ldots,s_{n(\boldsymbol s)}\}$. Similarly, $\tilde A_{\boldsymbol t}$ is a product of terms $\tilde A_{t,j}$ with $t+j \in T\coloneqq \{t_1,\ldots,t_{n(\boldsymbol t)}\}$. By definition of $\mathscr O$, the sets $S$ and $T$ are disjoint. Hence, under assumption \ref{A2}, $\tilde A_{\boldsymbol s}$ and $\tilde A_{\boldsymbol t}$ are independent, see Remark~\ref{rem:independence}. Thus, using the non-negativity of the $A$'s and the independence assumptions,
\begin{align*}
\sum_{(\boldsymbol s,\boldsymbol t)\in\mathscr O}\E[\tilde A_{\boldsymbol s}\tilde B_{\boldsymbol s}\tilde A_{\boldsymbol t}\tilde B_{\boldsymbol t}] & =\E[B^2]+ \sum_{(\boldsymbol s,\boldsymbol t)\in\mathscr O\setminus ({\boldsymbol 0},{\boldsymbol 0})} \E[\tilde A_{\boldsymbol s}]\E[\tilde A_{\boldsymbol t}]\E[B]^2\\
& \le \E[B^2] + \sum_{\boldsymbol s,\boldsymbol t\in \mathscr T}  \E[\tilde A_{\boldsymbol s}]\E[\tilde A_{\boldsymbol t}]\E[B]^2\\
&=\E[B^2] + \E[X]^2.
\end{align*}
By assumptions (a) and (c), we have $\phi_1(1) < 0$ and $\E[B^2]<\infty$, hence $\E[B]<\infty$ and therefore $\E[X]<\infty$ by Theorem~\ref{th:first_moment}. Therefore, the last line in the above display is finite, which finishes the proof in the case \ref{A2}.

We now treat the more delicate case \ref{A1}. Let again $(\boldsymbol s,\boldsymbol t)\in\mathscr O$. We decompose $\boldsymbol s$ and $\boldsymbol t$ according to their first jump (if it exists): $\boldsymbol s = (i)\boldsymbol s'$ and $\boldsymbol t = (j)\boldsymbol t'$, with $i,j\in\N_0$ and $\boldsymbol s',\boldsymbol t'\in \mathscr T$, and with the short-hand notation
\[
(i) = \begin{cases}
(0,i) & \text{if $i\ge 1$}\\
(0) & \text{if $i=0$.}
\end{cases}
\]
Using \ref{A1} and the definition of $\mathscr O$, we get, setting $A_0 = \tilde A_{\boldsymbol 0} = 1$,
\[
\E[\tilde A_{\boldsymbol s}\tilde A_{\boldsymbol t}] = \E[A_iA_j]\E[\tilde A_{\boldsymbol s'}]\E[\tilde A_{\boldsymbol t'}].
\]
Furthermore, the map
\[
\begin{cases}
\mathscr O \to \N_0^2\times \mathscr T^2\\
(\boldsymbol s,\boldsymbol t) \mapsto (i,j,\boldsymbol s',\boldsymbol t')
\end{cases}
\]
is obviously injective. Hence, by non-negativity of  the $A$'s and the independence assumptions, we get
\begin{align}
\nonumber
\sum_{(\boldsymbol s,\boldsymbol t)\in\mathscr O}\E[\tilde A_{\boldsymbol s}\tilde B_{\boldsymbol s}\tilde A_{\boldsymbol t}\tilde B_{\boldsymbol t}] &\le \E[B^2]\sum_{i,j\in \N_0}\sum_{\boldsymbol s',\boldsymbol t'\in \mathscr T} \E[A_iA_j]\E[\tilde A_{\boldsymbol s'}]\E[\tilde A_{\boldsymbol t'}]\\
\nonumber
&=\E[B^2] \left(\sum_{i,j\in \N_0} \E[A_iA_j]\right)\left(\sum_{\boldsymbol t\in \mathscr T} \E[\tilde A_{\boldsymbol t}]\right)^2\\
\label{eq:bientot_fini}
&=\E[B^2] \left(1+2e^{\phi_1(1)} + \sum_{i,j\ge1} \E[A_iA_j]\right)\E[X']^2\,,
\end{align}
{where $X'=\sum_{\boldsymbol t\in \mathscr T} \tilde A_{\boldsymbol t}$ is the generic element of a solution of the equation \eqref{eq:AR_infty_simple} with $B_t=1$ a.s. for all $t\in\Z$.}
By hypothesis (c), we have $\E[B^2]<\infty$.
Using hypotheses (a) and (b), one gets finiteness of the term in brackets on the RHS of \eqref{eq:bientot_fini}. Finiteness of $\E[X']$ follows from the hypothesis $\phi_1(1)<\infty$ and Theorem~\ref{th:first_moment} (2). Altogether, this yields $\sum_{(\boldsymbol s,\boldsymbol t)\in\mathscr O}\E[\tilde A_{\boldsymbol s}\tilde B_{\boldsymbol s}\tilde A_{\boldsymbol t}\tilde B_{\boldsymbol t}] < \infty$, which was to be proven.

We now treat the case $\theta \le 2$. By subadditivity, we have
\[
X_t^{\theta/2}\le \sum_{j=1}^{\infty}A_{t,j}^{\theta/2} X_{t-j}^{\theta/2}+B_t^{\theta/2}\,,\qquad t\in \Z\,.
\]
Thus $\E[X^\theta]\le \E[(X^{(\theta)})^2]$ where $X^{(\theta)}$ is {a generic element} of the non-anticipative stationary solution of 
\[
X_t^{(\theta)}= \sum_{j=1}^{\infty}A_{t,j}^{\theta/2} X_{t-j}^{(\theta)}+B_t^{\theta/2}\,,\qquad t\in \Z\,,
\]
and the general result for $\theta\le 2$ follow from the one for $\theta=2$ considering an alternative AR model.

The case $\theta \ge2$ on the other hand is treated by following the proof for $\theta =2$, but starting from the reverse inequality of \eqref{eq:second_moment_theta}, which holds by Minkowski's inequality, instead of \eqref{eq:second_moment}.

\end{proof}

\subsection{Extension to higher moments}
\label{sec:higher-moments}

It is natural to try to extend the above methods {to obtain necessary conditions and sufficient conditions for finiteness of higher moments}. One could define for every $k\ge 2$ two sets of closed and open $k$-tuples as follows (in the definition below, we identify $\boldsymbol t = (t_0,\ldots,t_n)\in\mathcal T$ with the set $\{t_0,\ldots,t_n\}$ and note that we have $0 \in \boldsymbol t$ for every $\boldsymbol t\in\mathscr T$):
\begin{align*}
\mathscr C^{(k)} &= \{(\boldsymbol t^1,\ldots,\boldsymbol t^k)\in \mathscr T^k: \boldsymbol t^1\cap\ldots\cap \boldsymbol t^k = \{0,m\}, \\
&\qquad\text{where }m = t^1_{n(\boldsymbol t^1)} = \cdots = t^k_{n(\boldsymbol t^k)}>0 \},\\
\mathscr O^{(k)} &= \{(\boldsymbol t^1,\ldots,\boldsymbol t^k)\in \mathscr T^k: \boldsymbol t^1\cap\ldots\cap \boldsymbol t^k = \{0\} \}.
\end{align*}
One can then again uniquely write any $k$-tuple $(\boldsymbol t^1,\ldots,\boldsymbol t^k)\in\mathcal T^k$ as a concatenation of a finite number of closed tuples and one open tuple. Moreover, an analogue of \eqref{eq:second_moment} holds: Assuming that $B =1$ a.s. for simplicity,  we have

\begin{align}
\label{eq:general_moment}
\E[X^k] = \left(\sum_{(\boldsymbol t^1,\ldots,\boldsymbol t^k)\in\mathscr O^{(k)}}\E[\tilde A_{\boldsymbol t^1} \cdots \tilde A_{\boldsymbol t^k}]\right) \sum_{n=0}^\infty \left(\sum_{(\boldsymbol t^1,\ldots,\boldsymbol t^k)\in\mathscr C^{(k)}}\E[\tilde A_{\boldsymbol t^1} \cdots \tilde A_{\boldsymbol t^k}]\right)^n .
\end{align}

A natural guess for a sufficient  condition for the finiteness of $\E[X^k]$ is that $$\sum_{(\boldsymbol t^1,\ldots,\boldsymbol t^k)\in\mathscr C^{(k)}}\E[\tilde A_{\boldsymbol t^1} \cdots \tilde A_{\boldsymbol t^k}] < 1,$$ plus probably some additional conditions on moments of lower order. However, we were not able to prove this. Indeed, while it is immediate that the above condition is \emph{necessary} for the finiteness of $\E[X^k]$ (for, otherwise, the sum over $n$ on the right-hand side of \eqref{eq:general_moment} is infinite), we were not able to prove sufficiency. The problem is caused by the term on the right-hand side of \eqref{eq:general_moment} involving $\mathscr O^{(k)}$: we were not able to give simple hypotheses under which this term is finite, as it is not clear how to generalize the proof for $k=2$ to general $k$. Indeed, while in the case of $k=2$, $\tilde A_{\boldsymbol s}$ and $\tilde A_{\boldsymbol t}$ are independent for $(\boldsymbol s, \boldsymbol t)\in\mathcal O^{(2)}$ (at least under assumption $\ref{A2}$), this is not the case anymore for $k > 2$. We believe that the extension of our methods to higher moments is an interesting question for further research and hope to be able to address it in the future.

\section{Second moment -- comparison to a classical criterion}

There exists an extensive statistics literature on the random coefficient autoregressive equation \eqref{eq:AR_infty}  in the case $p<\infty$, see for instance the book of \cite{Nicholls1982}. Here we consider autoregressive equations with a non-negative noise $B_t$, $t\in \Z$, under assumption \ref{A1}, a setting included in the extension of  \cite{Nicholls1982} due to \cite{tuan1986mixing}. The classical criterion \eqref{eq:nqcond} is necessary and sufficient for the stationary solution to have finite second moment. {In this section, we show that this criterion is actually equivalent to the criterion from Theorems \ref{th:second_moment}.}

{Throughout the section, we are in the context from the beginning of Section~\ref{sec:first_moment}, i.e. we consider \eqref{eq:AR_infty_simple} with non-negative coefficients $((A_{t,j})_j,B_t)_{t\in\Z}$  and we assume that \ref{A1} holds, i.e. $((A_{t,j})_j,B_t)_{t\in\Z}$ are iid copies of $((A_j)_{j\ge 1},B)$. We further assume that the equation is of finite order $p$, i.e.
\begin{enumerate}
\item[{\bf (p)}] There exists $p<\infty$, such that $A_{j} = 0$ a.s. for all $t\in\Z$ and all $j>p$.
\end{enumerate}
For simplicity, we will furthermore assume the following:
\begin{enumerate}
\item[{\bf (+)}] $A_{j} > 0$ a.s. for all $t\in\Z$, $1\le j\le p$. 
\end{enumerate}

Recall the definition of the matrices
\begin{equation}\label{eq:companion}
\boldsymbol A_t =
\begin{pmatrix}
A_{t,1} & A_{t,2} & \cdots & \cdots  & A_{t,p}\\
1 & 0 & \cdots & \cdots &0 \\
 0& \ddots& \ddots &  &\vdots\\ 
 \vdots&  \ddots& \ddots & \ddots &\vdots\\
 0& \cdots &0 & 1 &0
\end{pmatrix},
\quad 
\boldsymbol A =
\begin{pmatrix}
A_{1} & A_{2} & \cdots & \cdots  & A_{p}\\
1 & 0 & \cdots & \cdots &0 \\
 0& \ddots& \ddots &  &\vdots\\ 
 \vdots&  \ddots& \ddots & \ddots &\vdots\\
 0& \cdots &0 & 1 &0
\end{pmatrix}.
\end{equation}
}
We denote by $\otimes$ the tensor (or Kronecker) product of matrices. Recall that $\otimes$ is bilinear and associative and that it satisfies the \emph{mixed product property} with respect to the usual matrix product:
\[
(\boldsymbol M\otimes \boldsymbol N)(\boldsymbol P\otimes \boldsymbol Q) = \boldsymbol M \boldsymbol P\otimes \boldsymbol N \boldsymbol Q.
\]
For a square matrix $\boldsymbol M$, we further denote by $\rho(\boldsymbol M)$ its spectral radius, i.e.~the largest absolute value of its eigenvalues. We furthermore denote by $\boldsymbol I$ the identity matrix of dimension $p$.

\begin{theorem}\label{th:comparison}
{Assume \ref{A1}, {\bf (p)} and {\bf (+)}.} 
Then $\sum_{(\boldsymbol s,\boldsymbol t)\in\mathscr C}\E[\tilde A_{\boldsymbol s}\tilde A_{\boldsymbol t}] < 1$ and $\sum_{j=1}^\infty \E[A_j]<1$ if and only if $\rho(\E[\boldsymbol A \otimes \boldsymbol A]) < 1$.
\end{theorem}

\begin{proof}
{The statement of the theorem is independent of $B$, hence we can assume with no loss of generality that $B=1$ a.s.} By Theorem \ref{th:second_moment}, we have $\sum_{(\boldsymbol s,\boldsymbol t)\in\mathscr C}\E[\tilde A_{\boldsymbol s}\tilde A_{\boldsymbol t}] < 1$ and $\sum_{j=1}^\infty \E[A_j]<1$ if and only if $\E[X^2] < \infty$ if and only if $\E[{\widetilde X}^2] < \infty$, where we recall that $\tilde X$ is defined in \eqref{eq:def_Xtilde2}. It follows from this definition that
\[
\widetilde X =  \Big(\sum_{n\ge0} \boldsymbol A_0\cdots \boldsymbol A_{-n+1}\Big)_{11} \eqd \Big(\sum_{n\ge0} \boldsymbol A_1\cdots \boldsymbol A_n\Big)_{11},
\]
where the equality in law follows from \ref{A1}. Hence, defining the matrix
\[
\boldsymbol M\coloneqq\E\left[\left(\sum_{n\ge0} \boldsymbol A_1\cdots \boldsymbol A_n\right) \otimes \left(\sum_{n\ge0} \boldsymbol A_1\cdots \boldsymbol A_n\right)\right]\,,
\]
and denoting by $(\boldsymbol M_{ij})_{i,j=1,\ldots,p^2}$ its coordinates, we have
\[
\E[{\widetilde X}^2] < \infty \iff \boldsymbol M_{11} < \infty.
\]

We now claim that $\boldsymbol M_{11}$ is finite if and only if $\rho(\E[\boldsymbol A\otimes \boldsymbol A]) < 1$. To see this, we first express the matrix $\boldsymbol M$ as follows:
\begin{align*}
\boldsymbol M&=\E\left[\left(\sum_{n\ge0} \boldsymbol A_1\cdots \boldsymbol A_n\right) \otimes \left(\sum_{n\ge0} \boldsymbol A_1\cdots \boldsymbol A_n\right)\right]\\
&= \E\left[\sum_{n,m\ge0} \boldsymbol A_1\cdots \boldsymbol A_n \otimes \boldsymbol A_1\cdots \boldsymbol A_m\right]\\
&= \E\Big[\sum_{n\ge0} (\boldsymbol A_1\otimes \boldsymbol A_1)\cdots (\boldsymbol A_n\otimes \boldsymbol A_n)\Big(\boldsymbol I\otimes \boldsymbol I\\
&\qquad+ \sum_{k\ge1} \big((\boldsymbol I \otimes \boldsymbol A_{n+1})\cdots (\boldsymbol I \otimes \boldsymbol A_{n+k})+(\boldsymbol A_{n+1} \otimes \boldsymbol I)\cdots (\boldsymbol A_{n+k} \otimes \boldsymbol I)\big) \Big)\Big]\\
&= \left(\sum_{n\ge0} \E[\boldsymbol A\otimes \boldsymbol A]^n\right) \left(\boldsymbol I\otimes \boldsymbol I + \sum_{k\ge1} \big(\E[\boldsymbol I\otimes \boldsymbol A]^k +  \E[\boldsymbol A\otimes \boldsymbol I]^k\big)\right)\\
&= \left(\sum_{n\ge0} \E[\boldsymbol A\otimes \boldsymbol A]^n\right) \left(\boldsymbol I\otimes \boldsymbol I + \boldsymbol I\otimes  \left(\sum_{k\ge1} \E[\boldsymbol A]^k\right) +  \left(\sum_{k\ge1} \E[\boldsymbol A]^k\right)\otimes \boldsymbol I\right)\,.
\end{align*}
Now, note that $\rho(\E[\boldsymbol A]) \le \sqrt{\rho(\E[\boldsymbol A\otimes \boldsymbol A])}$ by Lemma~\ref{lem:cauchyschwarz_tensor} below. Hence, if $\rho(\E[\boldsymbol A\otimes \boldsymbol A]) < 1$, all the above series are finite. Thus all the entries of $\boldsymbol M$ are finite, in particular $\boldsymbol M_{11}$.

On the other hand, suppose $\rho(\E[\boldsymbol A\otimes \boldsymbol A])\ge1$; we will show that the first series is infinite. 
Note that, almost surely, the matrix $\boldsymbol A$ is irreducible and aperiodic by assumption {\bf (+)}. Hence, $\boldsymbol A^k > 0$ entrywise for $k$ large enough, i.e.~the matrix is almost surely primitive in the language of \cite{seneta2006non}. Since $(\boldsymbol A\otimes \boldsymbol A)^k = \boldsymbol A^k\otimes \boldsymbol A^k$, it follows that $\boldsymbol A\otimes \boldsymbol A$ is almost surely primitive as well, hence $\E[\boldsymbol A\otimes\boldsymbol A]$ is primitive by non-negativity. We can therefore apply the strong version of the Perron-Frobenius theorem for primitive matrices (see Theorem 1.2 of \cite{seneta2006non})  to the  matrix $\E[\boldsymbol A\otimes \boldsymbol A]$. It follows that there exist right and left eigenvectors  $\boldsymbol u$ and $\boldsymbol v$, respectively, with positive entries and associated to the largest eigenvalue $\rho = \rho(\E[\boldsymbol A\otimes \boldsymbol A]) \ge 1$ so that
\[
 \rho^{-n}\E[\boldsymbol A\otimes \boldsymbol A]^n  = \boldsymbol u\boldsymbol v^T +o(1) \,,\qquad  n\to \infty\,.
\]
One deduces that, entrywise,
\[
\sum_{n=0}^m\E[\boldsymbol A\otimes \boldsymbol A]^n \ge \sum_{n=0}^m\rho^{-n}\E[\boldsymbol A\otimes \boldsymbol A]^n  = m\boldsymbol u\boldsymbol v^T +o(m)\,,\qquad m\to\infty\,.
\]
Since all the entries of $\boldsymbol u\boldsymbol v^T$ are positive, this yields for every $m\ge0$,
\[
\E[\tilde X^2]=\boldsymbol M_{11}\ge\left( \sum_{n=0}^m \E[\boldsymbol A\otimes \boldsymbol A]^n\right)_{11}\ge m\boldsymbol u_1\boldsymbol v_1 +o(m),
\]
which implies that $\E[\tilde X^2]=\infty$.
This proves the result.
\end{proof}

The following lemma was needed in the proof of Theorem~\ref{th:comparison} above and is included here for completeness.

\begin{lemma}
\label{lem:cauchyschwarz_tensor}
Let $\boldsymbol A$ be a random $m\times m$ companion matrix as in \eqref{eq:companion} with non-negative entries, $m\ge 1$, such that $\E[\boldsymbol A]$ is irreducible. Then,
$\rho(\E[\boldsymbol A]) \le \sqrt{\rho(\E[\boldsymbol A\otimes\boldsymbol  A])}$.
\end{lemma}
\begin{proof}
Applying Perron-Frobenius theorem to the irreducible companion matrix $\E[\boldsymbol A]$, the spectral radius $\lambda=\rho(\E[\boldsymbol A])$ is the eigenvalue associated to the eigenvector with positive entries $\boldsymbol v=(\lambda^{m-1},\lambda^{m-2},\ldots,\lambda,1)^T$. Then $(\boldsymbol A\otimes\boldsymbol  A)(\boldsymbol v\otimes\boldsymbol  v)$ is a vector of $\R^{m^2}$ whose entries are affine functions of the entries of $\boldsymbol A$ except the first one, which is equal to$(\boldsymbol v^T\boldsymbol A_{1\cdot})^2$, where $\boldsymbol A_{1\cdot}$ denotes the first row of $\boldsymbol A$. Jensen's inequality yielding 
\[
\E[(\boldsymbol v^T\boldsymbol A_{1\cdot})^2]\ge (\boldsymbol v^T\boldsymbol \E[\boldsymbol A_{1\cdot}])^2\,,
\]
we obtain that 
\begin{align*}
\E[(\boldsymbol A\otimes\boldsymbol  A)](\boldsymbol v\otimes\boldsymbol  v)&\ge (\E[\boldsymbol A]\otimes\E[\boldsymbol  A])(\boldsymbol v\otimes\boldsymbol  v)\\
&= (\E[\boldsymbol A]\boldsymbol v\otimes\E[\boldsymbol  A]\boldsymbol v)\\
&= (\lambda \boldsymbol v\otimes\lambda \boldsymbol v)\\
&=\lambda^2 ( \boldsymbol v\otimes \boldsymbol v)\,.
\end{align*}
From the facts that the entries of the vector $( \boldsymbol v\otimes \boldsymbol v)$ are positive, that the matrix $\E[(\boldsymbol A\otimes\boldsymbol  A)]$ has non-negative entries and is irreducible, the dual of the Subinvariance Theorem (see for instance Exercise 1.17 of \cite{seneta2006non}) implies that $\lambda^2\le \rho(\E[(\boldsymbol A\otimes\boldsymbol  A)])$. The desired result follows.
\end{proof}

\section{Examples}

In this section we work out two toy examples. The special feature of these examples is that explicit necessary and sufficient conditions for finiteness of moments are available. Note that necessary conditions are rarely discussed in the literature. Two notable exceptions are \cite{douc2008existence} for the ARCH model and \cite{giraitis2004larch} for the LARCH model. The first one has positive coefficients as in the present article. Using our approach, we recover and extend the results of \cite{douc2008existence}.

\subsection{A simple example of order 2.}\label{sec:example1}

We consider the case where $A_1,A_2$ are iid copies of a r.v. $A$, $\E[A] < 1$, and $A_k=0$ for $k\ge 3$. Furthermore, we assume that condition \ref{A1} is satisfied. We wish to make explicit the criterion for finiteness of second moments from Theorem \ref{th:second_moment} and compare it to the classical criterion \eqref{eq:nqcond} from \cite[Corollary 2.2.2]{Nicholls1982}. By Theorem~\ref{th:comparison} these are equivalent.

The set of closed pairs $\mathscr C$ is very simple: if $(\boldsymbol s,\boldsymbol t)\in\mathscr C$, then either
\begin{enumerate}
\item $(\boldsymbol s,\boldsymbol t) \in \{((0,1),(0,1)), ((0,2),(0,2))\}$, or
\item there exists $k\ge1$, such that $\boldsymbol s$ starts with a jump of size 1, then $\boldsymbol t$ and $\boldsymbol s$ alternately do jumps of size 2 ($k$ times in total), followed by a last jump of size $1$ (from either $\boldsymbol s$ or $\boldsymbol t$ depending on whether $k$ is odd or even), or
\item same as the previous case, but with the roles of $\boldsymbol s$ and $\boldsymbol t$ exchanged.
\end{enumerate}
If $(\boldsymbol s,\boldsymbol t)$ is of the first type, then
\[
\E[A_{\boldsymbol s}A_{\boldsymbol t}] = \E[A^2].
\]
If $(\boldsymbol s,\boldsymbol t)$ is of the second or third type, with the corresponding $k$, then
\[
\E[A_{\boldsymbol s}A_{\boldsymbol t}] = \E[A]^{k+2}.
\]
Hence,
\[
\sum_{(\boldsymbol s,\boldsymbol t)\in\mathscr C} \E[A_{\boldsymbol s}A_{\boldsymbol t}] = 2\E[A^2] + 2 \sum_{k\ge 1} \E[A]^{k+2} = 2\left(\E[A^2] + \frac{\E[A]^3}{1-\E[A]}\right).
\]
Writing $a=\E[A]$ and $b=\E[A^2]$, this gives
\begin{align}
\sum_{(\boldsymbol s,\boldsymbol t)\in\mathscr C} \E[A_{\boldsymbol s}A_{\boldsymbol t}] \ge 1
\nonumber &\iff 2\left(b+\frac{a^3}{1-a}\right) \ge 1\\
\label{eq:cond_1}&\iff (2b-1)(1-a)+2a^3 \ge 0.
\end{align}

We now calculate $\E[\boldsymbol A\otimes\boldsymbol A]$. First, we have
\[
\boldsymbol A\otimes \boldsymbol A =
\begin{pmatrix}
A_1^2 & A_1A_2 & A_2A_1 & A_2^2\\
A_1 & 0 & A_2 & 0\\
A_1 & A_2 & 0 & 0\\
1 & 0 & 0 & 0
\end{pmatrix},
\]
and so
\[
\E[\boldsymbol A\otimes \boldsymbol A] = 
\begin{pmatrix}
b & a^2 & a^2  & b\\
a & 0 & a & 0\\
a & a & 0 & 0\\
1 & 0 & 0 & 0
\end{pmatrix}.
\]
The characteristic polynomial of this matrix is easily calculated to be
\[
\det(\E[\boldsymbol A\otimes \boldsymbol A] - X(\boldsymbol I\otimes \boldsymbol I)) = (X+a)P(X),
\]
where the polynomial $P(X)$ is defined as
$$P(X) = X^3 - (a+b)X^2+ (-2a^2-b+ab)X+ab.$$
Hence, the eigenvalues of the matrix $\E[\boldsymbol A\otimes \boldsymbol A]$ are $-a$ as well as the roots of the degree-3 polynomial $P(X)$.
Recall that $a\in (0,1)$. We have
\begin{align*}
P(-1) &= -1 - a + 2a^3 < 0\\
P(0) &= ab > 0\\
P(a) &= -2 a^4 < 0.
\end{align*}
Therefore, the smallest root of the polynomial $P$ is greater than $-1$ and, furthermore, the largest root is greater or equal than $1$ if and only if $P(1) \le 0$. Hence,
\begin{align*}
\rho(\E[\boldsymbol A\otimes \boldsymbol A]) \ge 1 \iff P(1) \le 0 \iff (2b-1)(1-a) + 2a^3 \ge 0,
\end{align*}
which is exactly \eqref{eq:cond_1}, as expected.

\subsection{A degenerate infinite memory case.}
We consider the following example motivated by the GARCH(1,1) model that also admits an ARCH($\infty$) representation. Consider the infinite memory recursion
\begin{equation}
\label{eq:GARCH}
X_t= \dfrac1{1-\beta}+ \sum_{k\ge 0 }\beta^{k+1}Z_{t-k}X_{t-1-k}\,,\qquad t\in \Z\,,
\end{equation}
where $(Z_t)_{t\in\Z}$ is an iid sequence of copies of a non-negative random variable $Z$. Setting $A_{t,j} = \beta^j Z_{t-j+1}$, $j\ge1$, $t\in \Z$,we see that \eqref{eq:GARCH} is of the form \eqref{eq:AR_infty} with this choice of $A_{t,j}$. Furthermore, assumption \ref{A2} is verified.

On the other hand, it is well-known and can easily be checked that the stationary solution of the above equation, if it exists, also satisfies the Markov equation 
\begin{equation}\label{eq:srevol}
  X_t= 1+ \beta(1+ Z_{t})X_{t-1}\,,\qquad t\in \Z\,. 
\end{equation}
As before, denote by $X$ the limit in law of $X_t$ as $t\to\infty$.
From the latter recursion, one can obtain that, for every $\theta > 0$,
\begin{equation}
\label{eq:garch_theta}
\E[X^\theta]<\infty \iff \E[\beta^\theta (1+ Z)^\theta]<1.
\end{equation}
In fact, it is known that there exists  $C>0$ such that $\P(X>x)\sim C/x^\alpha$ as $x\to\infty$ for $\alpha>0$ satisfying the equation $\E[\beta^\alpha (1+ Z)^\alpha]=1$, see \cite{BDM} for more details.
Comparing the necessary and sufficient condition \eqref{eq:garch_theta} with the conditions obtained by computing the functions $\phi_1$, $\tilde \phi_1$, $\phi_2$ and $\tilde \phi_2$ respectively defined in \eqref{eq:phi_1}, \eqref{eq:phi_1_tilde}, \eqref{eq:phi_2} and \eqref{eq:phi_2_tilde}, we thus get an explicit benchmark for our conditions.

The functions $\phi_1$ and $\tilde \phi_1$ are easily calculated from  \eqref{eq:GARCH}, which yields,
\begin{align*}
    \phi_1(\theta)&=\log\Big(\dfrac{\beta^\theta \E[Z^\theta]}{1-\beta^\theta}\Big)\,,\\
    \tilde \phi_1(\theta)&=\log\Big(\dfrac{\beta \E[Z^\theta]^{1/\theta}}{1-\beta}\Big)\,.
\end{align*}

We now calculate $\phi_2$. For $\boldsymbol s = (s_0,\ldots,s_n) \in \mathscr T$, write $Z_{\boldsymbol s} = Z_{s_1}\cdots Z_{s_n}.$ Then it is easily checked that 
\begin{align*}
\sum_{(\boldsymbol s,\boldsymbol t)\in\mathscr C} \E[\tilde A^{\theta/2}_{\boldsymbol s}\tilde A^{\theta/2}_{\boldsymbol t}]&=\sum_{(\boldsymbol s,\boldsymbol t)\in\mathscr C}\beta^{\theta/2(s_{n(\boldsymbol s)}+t_{n(\boldsymbol t)})}\E[Z^{\theta/2}_{\boldsymbol s}Z^{\theta/2}_{\boldsymbol t}].
\end{align*}
Setting
\[
\mathscr C_m = \{(\boldsymbol s,\boldsymbol t)\in\mathscr C:s_{n(\boldsymbol s)}=t_{n(\boldsymbol t)}=m\}\,,\quad m\ge1\,,
\]
the above equality is expressed as
\[
\sum_{(\boldsymbol s,\boldsymbol t)\in\mathscr C}
\E[\tilde A^{\theta/2}_{\boldsymbol s}\tilde A^{\theta/2}_{\boldsymbol t}]=\sum_{m=1}^\infty
\beta^{\theta m} \sum_{(\boldsymbol s,\boldsymbol t)\in\mathscr C_m}\E[Z^{\theta/2}_{\boldsymbol s}Z^{\theta/2}_{\boldsymbol t}].
\]
Fix $m\ge1$ and let $(\boldsymbol s,\boldsymbol t)\in\mathscr C_m$. By definition, we have
\[
 \E[Z^{\theta/2}_{\boldsymbol s}Z^{\theta/2}_{\boldsymbol t}]=\E[Z^\theta]\E[Z^{\theta/2}]^{n(\boldsymbol s)+n(\boldsymbol t)-2}.
\]
Note that the exponent $j = n(\boldsymbol s)+n(\boldsymbol t)-2$ is equal to the number of points in $\{1,\ldots,m-1\}$ which are contained in either $\boldsymbol s$ or $\boldsymbol t$. There are $\binom{m-1}{j}$ ways to choose these points and, given the choice of these points there are $2^{j}$ ways to distribute them among $\boldsymbol s$ and $\boldsymbol t$. It follows that
\begin{align*}
\sum_{(\boldsymbol s,\boldsymbol t)\in\mathscr C_m}\E[Z^{\theta/2}_{\boldsymbol s}Z^{\theta/2}_{\boldsymbol t}]
&= \E[Z^\theta] \sum_{j=0}^{m-1} \binom{m-1}j (2\E[Z^{\theta/2}])^j\\
&= \E[Z^\theta] (1+2\E[Z^{\theta/2}])^{m-1}.
\end{align*}
Summing over $m$ yields
\[
\sum_{(\boldsymbol s,\boldsymbol t)\in\mathscr C}
\E[\tilde A^{\theta/2}_{\boldsymbol s}\tilde A^{\theta/2}_{\boldsymbol t}] = \E[Z^\theta] \sum_{m=1}^\infty \beta^{\theta m} (1+2\E[Z^{\theta/2}])^{m-1}.
\]
We obtain
\[
\phi_2(\theta)=\begin{cases}
\log\Big(\dfrac{\beta^\theta \E[Z^{\theta} ]}{1-\beta^\theta(1+2\E[Z^{\theta/2} ])}\Big), &\text{if } \beta^\theta(1+2\E[Z^{\theta/2} ]) < 1\\
+\infty, & \text{otherwise.}
\end{cases}
\]
A similar argument shows that 
\[
\tilde \phi_2(\theta)=
\begin{cases}
\log\Big(\dfrac{\beta^2 \E[Z^{\theta} ]^{2/\theta}}{1-\beta^2(1+2\E[Z^{\theta/2} ]^{2/\theta})}\Big), & \text{if }\beta^2(1+2\E[Z^{\theta/2} ]^{2/\theta})<1\\
+\infty, & \text{otherwise.}
\end{cases}
\]

With the explicit expressions of $\phi_1$, $\tilde \phi_1$, $\phi_2$ and $\tilde \phi_2$ at hand, we can now compare the conditions for finiteness of $\E[X^\theta]$ provided by Theorem~\ref{th:first_moment} and Theorem~\ref{th:second_moment} with the condition \eqref{eq:garch_theta}. To do this, for every $\theta \in (0,3]$, we consider the critical value $\beta_\theta$, such that $\E[X^\theta]$ is finite for $\beta < \beta_\theta$ and infinite for $\beta > \beta_\theta$ (the existence and uniqueness of $\beta_\theta$ is easily seen by monotonicity in $\beta$ of the involved functions). Our theorems provide upper and lower bounds in the phases $\theta\in(0,1]$, $\theta\in[1,2]$ and $\theta\in[2,3]$, which are furthermore sharp for $\theta\in\{1,2\}$. In Figure \ref{fig:phis}, we compare these bounds with the exact value for $\beta_\theta$ obtained from the equation $\log(\E[\beta^\theta (1+ Z)^\theta])=0$ with $Z$ being $\chi_1^2$-distributed. One can notice that the use of second moment methods allows to greatly improve the quality of the bounds obtained by first moment methods, as soon as $\theta > 1.2$.

\begin{figure}[ht]
    \centering
    \includegraphics[height=10cm]{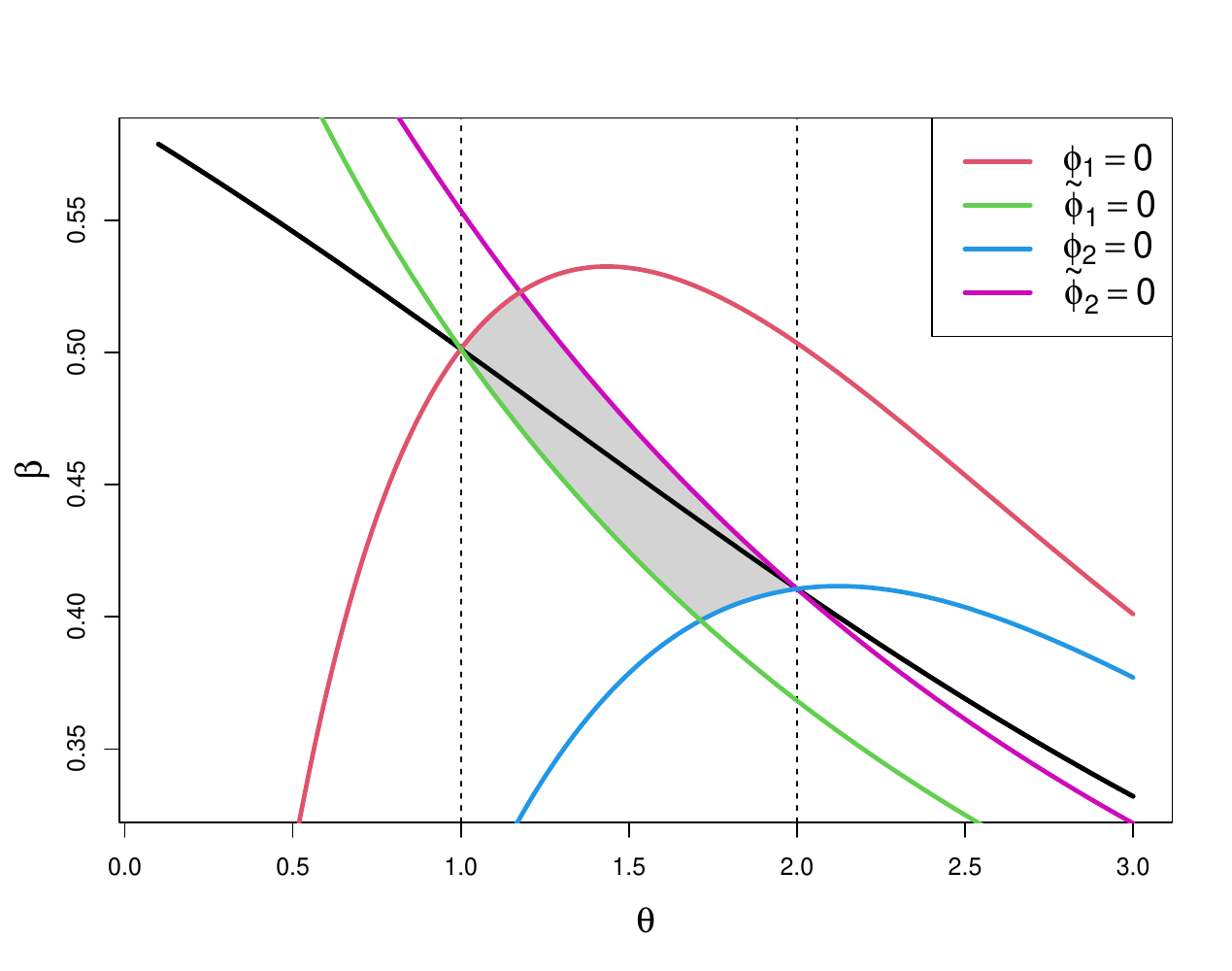}
    \caption{Illustration of our necessary and sufficient conditions of moments {applied to the marginal solution $X(\beta)$ of the AR($\infty$) model $X_t= (1-\beta)^{-1}+ \sum_{k\ge 0 }\beta^{k+1}Z_{t-k}X_{t-1-k}$, $t\in\Z$, $(Z_t)$ iid $\chi_1^2$-distributed, $0<\beta<1$, see \eqref{eq:GARCH}. The black curve corresponds to the root $\beta$ of the equation $\E[(1+\beta Z_0)^{\theta}]=1$, $\theta>0$, such that $\E[X(\beta)^\theta]<\infty$ if and only if $\beta$ lies below it}. Our sufficient conditions of moments correspond to the curves in red (for $0<\theta\le 1$), blue (for $0<\theta\le 2$, green (for $\theta\ge 1$) and purple (for $\theta\ge 2$).  Our necessary conditions of moments correspond to the curves in green (for $0<\theta\le 1$), purple (for $0<\theta\le 2$), red (for $\theta\ge 1$) and blue (for $\theta\ge 2$). {The grey aera corresponds to the pairs $(\theta,\beta)$ for which our approach cannot distinguish whether $\E[X(\beta)^\theta]$ is finite or not, $1\le \theta\le 2$.}}
    \label{fig:phis}
\end{figure}

Conditions of moments are crucial for proving the asymptotic normality of estimators of the parameters driving time series models. In particular for  GARCH models it is important to assume the sharpest condition of moments in order to apply the estimators on plausibly heavy-tailed financial time series; see \cite{francq2019garch} for a reference textbook on GARCH modeling. For instance
 \cite{bardet2009asymptotic} used second moment methods in infinite memory processes in order to prove the asymptotic normality of the Quasi-Maximum Likelihood Estimator (QMLE) of the parameters $\theta=(\omega,\alpha_1,\beta_1)$ in the GARCH(1,1) model $\vep_t=Z_t \sigma_t$ for iid $(Z_t)$ and $(\sigma_t)$ satisfying the recursion
\[
\sigma_t^2=\omega + \alpha_1 \vep_{t-1}^2+\beta_1 \sigma_{t-1}^2\,,\qquad t\in \Z\,.
\]
This equation coincides with \eqref{eq:srevol} for $X_t=\sigma_t^2/\omega$, $\beta=\beta_1$ and $Z_t=\alpha_1Z_t^2/\beta_1$. Then the necessary and sufficient condition of moments of order $2$ on $\sigma_t^2$ (and then $4$ on $\vep_t$) is
\[
\phi_2(2)=\tilde \phi_2(2)<0\Longleftrightarrow \alpha^2_1\E[Z^4]+2\alpha_1\beta_1\E[Z^2]+\beta_1^2<1\,.
\]
However the sufficient condition used in  \cite{bardet2009asymptotic} is 
\[
\phi_1(2)<0\Longleftrightarrow \alpha_1\E[Z^4]^{1/2}+\beta_1<1\,.
\]
In Figure \ref{fig:garch} we illustrate both conditions on the coefficients $(\alpha_1,\beta_1)$ for standard gaussian $Z$. Note however that the asymptotic normality of the QMLE holds under much weaker log-moments conditions than the second moment condition $\phi_2(2)=\tilde \phi_2(2)<0$ by using directly the recursion \eqref{eq:srevol} rather than \eqref{eq:GARCH}; see  \cite{francq2019garch}. The infinite memory approach is only required in more complex settings such as in \cite{bardet2022contrast} for time varying parameters GARCH(1,1) models. There, the best known condition for asymptotic normality is { the analog of $\phi_1(2)<0$} and our approach might improve on the existing literature.
\begin{figure}
    \centering
    \includegraphics[height=10cm]{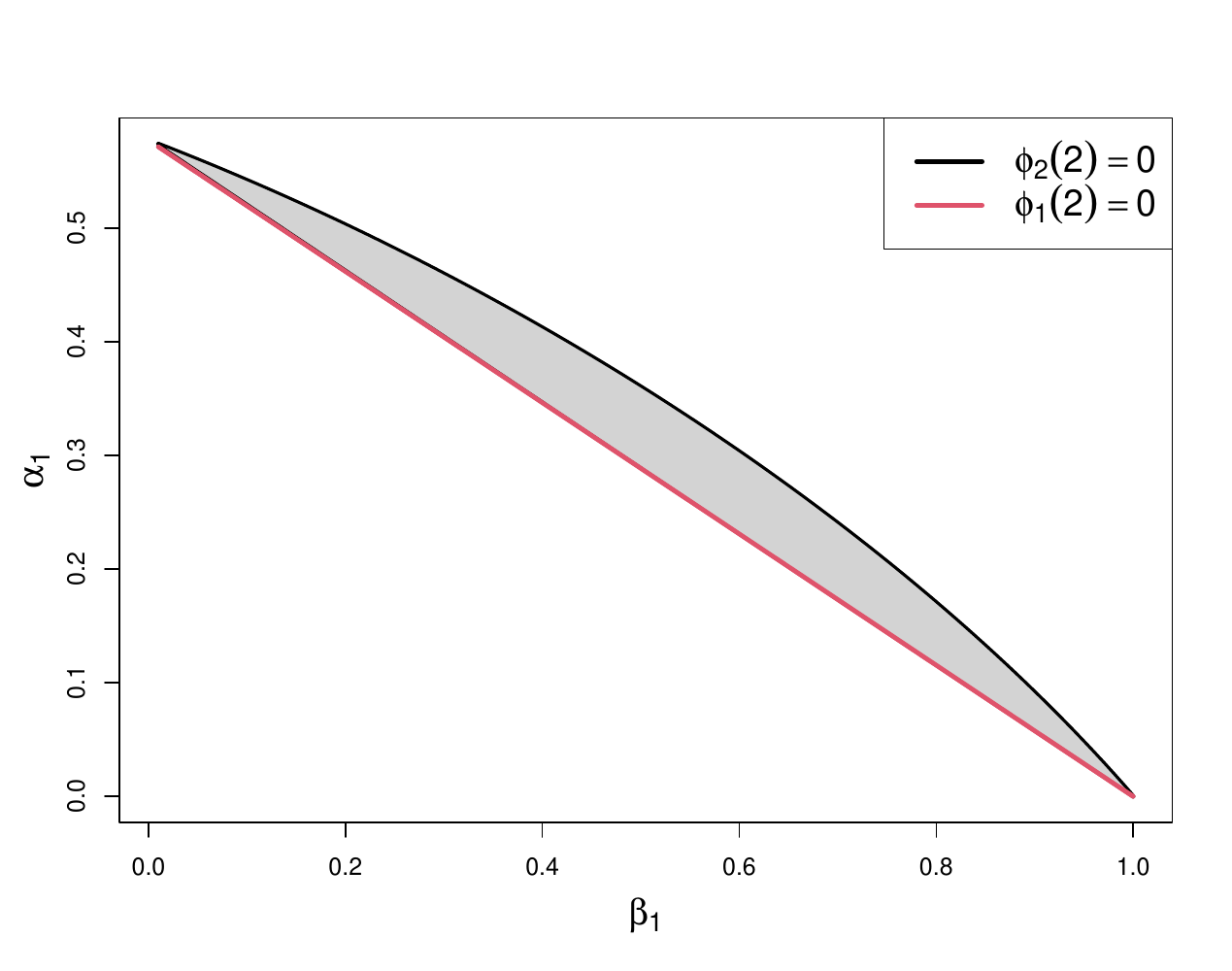}
    \caption{Illustration of sufficient conditions {for the finiteness of second moments of the volatility $(\sigma_t^2)$ of the GARCH(1,1) model, solution to the equation $\sigma_t^2=\omega + (\alpha_1 Z_{t-1}^2+\beta_1) \sigma_{t-1}^2$, $t\in \Z$, $(Z_t)$ iid Gaussian, in terms of the coefficients $(\alpha_1,\beta_1)$.  The red line corresponds to the equation $\alpha_1=(1-\beta_1)/\E[Z^4]^{1/2}$, $0<\beta_1<1$, such that the parameters $(\alpha_1,\beta_1)$ below the line satisfy $\E[\sigma_t^4]<\infty$ by an application of the first moment approach. The black curve corresponds to the root $\alpha_1$ of the equation $\alpha^2_1\E[Z^4]+2\alpha_1\beta_1\E[Z^2]+\beta_1^2=1$, $0<\beta_1<1$, corresponding to our second moment approach. The grey area represents the parameters $(\alpha_1,\beta_1)$ satisfying $\E[\sigma_t^4]<\infty$ by an application of our paper but excluded when using previous infinite memory literature.}}
    \label{fig:garch}
\end{figure}

\bibliographystyle{abbrvnat}
\bibliography{2018_autoregressive}
\end{document}